\documentclass[12pt,reqno]{amsart}
\usepackage[utf8]{inputenc}
\usepackage{amsthm,amsmath,amsxtra,latexsym,amscd,amssymb,exscale,mathrsfs,mathtools,xparse}
\usepackage[abbrev]{amsrefs}
\usepackage{geometry}
\usepackage{boxedminipage}
\usepackage[bookmarks=true]{hyperref}
\usepackage{bookmark}
\newtheorem{theorem}{Theorem}[section]
\newtheorem{lemma}[theorem]{Lemma}
\newtheorem{proposition}[theorem]{Proposition}

\newtheorem*{corollary*}{Corollary}

\theoremstyle{definition}
 
\theoremstyle{remark}
\newtheorem{remark}{Remark}[section]
\theoremstyle{remark}

\theoremstyle{remark}
\newtheorem*{remark*}{Remark}
\numberwithin{equation}{section}

\let\div\undefined
\DeclareMathOperator{\div}{div}

\let\Real\undefined
\DeclareMathOperator{\Real}{\mathbb{R}}

\title[Monotonicity formulas for minimal submanifolds]{Monotonicity formulas for minimal submanifolds involving M\"obius transformations}
\author{Doanh Pham}
\address{Beijing International Center for Mathematical Research, Peking University, Beijing 100871, China}
\email{doanhpham@pku.edu.cn}

\calclayout

\begin{document}
	
	\maketitle
	
\begin{abstract}
For a minimal submanifold of the Euclidean space, we prove monotonicity formulas for its (weighted) volume within images of concentric balls under M\"obius transformations.
\end{abstract}

\section{Introduction}

In this paper, we let $B^n(a,r)$ and $S^{n-1}(a,r)$ denote the ball and sphere, respectively, in the Euclidean space $\Real^n$ centered at $a$ with radius $r$. For convenience, we denote $B^n_r = B^n(0,r)$ and $S^{n-1}_r = S^{n-1}(0,r) = \partial B^n_r$.

Minimal submanifolds are important objects in geometry. They are critical points of the area functional and characterized by having vanishing mean curvature. For a $k$-dimensional minimal submanifold $\Sigma$ of $\Real^n$, the classical monotonicity formula (see, for example, \cite{Simon_Lectures on geometric measure theory}) states that the function
$$r \mapsto r^{-k} |\Sigma \cap B^n_r|$$
is monotone increasing as long as $\partial\Sigma \cap B^n_r = \emptyset$. The formula is a useful tool to study singularities of minimal submanifolds and has various geometric implications. In particular, by taking the limit as $r \to 0$, it shows that if $\Sigma$ is a $k$-dimensional minimal submanifold of $B^n_1$ which contains the origin and satisfies $\partial\Sigma \subset S^{n-1}_1$, then $|\Sigma| \geq |B^k_1|$ with equality if and only if $\Sigma$ is totally geodesic.

The following `moving-center' monotonicity formulas for minimal submanifolds of $\Real^n$ were obtained by Zhu \cite{Zhu_JFA2018}, and Naff and Zhu \cite{Naff-Zhu_Monotonicity for minimal in constant cruvature_arXiv2022}:

\begin{theorem}[\cite{Zhu_JFA2018,Naff-Zhu_Monotonicity for minimal in constant cruvature_arXiv2022}] \label{Naff-Zhu Theorem} Fix $a \in B^n_1$ and denote the family of balls
$$E_s = B^n((1-s)a,\, \sqrt{s(1-|a|^2) + s^2|a|^2}) \quad \text{for } s \geq 0.$$
On the half space $\{x \in \Real^n: |a|^2 + 1 - 2\langle a,x \rangle > 0\}$, define the function
$$f(x) = \frac{|x-a|^2}{|a|^2 + 1-2\langle a,x \rangle}.$$
Suppose that $\Sigma$ is a $k$-dimensional minimal submanifold of $E_{\bar{s}}$ with $\partial\Sigma \subset \partial E_{\bar{s}}$ for some $\bar{s} > 0$. For $s \in (0, \bar{s})$, define
\begin{equation*}
Q_A(s) = s^{-\frac{k}{2}}|\Sigma \cap E_s| \quad \text{and} \quad Q_I(s) = s^{-\frac{k}{2}} \int_{\Sigma \cap E_s} \frac{|(x-a)^\top|^2}{|x-a|^2}.
\end{equation*}
Then for $0 < s < t < \bar{s}$, we have volume monotonicity
\begin{equation}\label{eq:Zhu monotonicity}
Q_A(t) - Q_A(s) = \int_{\Sigma \cap E_t \backslash E_s} f^{-\frac{k}{2}} \left( \frac{|(x-a)^\perp|^2 + f^2 |a^\top|^2}{|x-a|^2}\right),
\end{equation}
and weighted monotonicity
\begin{align}\label{eq: Naff-Zhu weighted monotonicity}
Q_I(t) - Q_I(s) = \int_{\Sigma \cap E_t \backslash E_s} f^{-\frac{k-4}{2}} \frac{|a^\top|^2}{|x-a|^2} + \frac{k}{2} \int_{s}^{t} \left(\tau^{-\frac{k+2}{2}} \int_{\Sigma \cap E_{\tau}}\frac{|(x-a)^\perp|^2}{|x-a|^2} \right) d\tau.
\end{align}
In particular, $Q_A$ and $Q_I$ are both monotone increasing on $(0, \bar{s})$. Furthermore, if one of them is constant, then $\Sigma$ is a totally geodesic disk containing $a$ and orthogonal to $a$.
\end{theorem}

In \eqref{eq:Zhu monotonicity} and throughout the rest of this paper, we use $X^\top$ to denote the tangential projection of a vector field $X$ to the tangent space $T_x\Sigma$ of $x \in \Sigma$, and $X^\perp$ to denote its orthogonal complement. Theorem \ref{Naff-Zhu Theorem} is a generalization of the classical monotonicity formula and consequently gives the sharp lower bound for volume of minimal submanifolds in the unit ball passing through a prescribed point. Specifically, by choosing $t = 1$ and taking the limit as $s \to 0$ in either \eqref{eq:Zhu monotonicity} or \eqref{eq: Naff-Zhu weighted monotonicity}, one recovers the following result which was previously proved by Alexander and Osserman \cite{Alexander-Osserman_1975} for simply connected surfaces and later by Brendle and Hung \cite{Brendle-Hung_GAFA2017} in full generality.

\begin{corollary*}[\cite{Brendle-Hung_GAFA2017}]
Suppose that $\Sigma$ is a $k$-dimensional minimal submanifold of $B^n_1$ which contains a point $a \in B^n_1$ and satisfies $\partial\Sigma \subset S^{n-1}_1$. Then $|\Sigma| \geq |B^k_1|(1 - |a|^2)^{\frac{k}{2}}$ with equality if and only if $\Sigma$ is totally geodesic and orthogonal to $a$.
\end{corollary*}
	
See \cite{Naff-Zhu_Prescribed point area estimate_arXiv2022, Naff-Zhu_Monotonicity for minimal in constant cruvature_arXiv2022} for related monotonicity formulas and prescribed point volume estimates for minimal submanifolds in space forms. In this paper, we formulate and prove  monotonicity formulas for minimal submanifolds of $\Real^n$ which involve M\"obius transformations. Let $\overline{\Real^n} = \Real^n \cup \{\infty\}$ denote the one-point compactification of $\Real^n$. By definition, a M\"obius transformation on $\overline{\Real^n}$ is a finite composition of reflections in spheres or planes. These transformations are characterized in the following theorem.

\begin{theorem}[see e.g. \cite{Beardon_Geometry of Discrete}]\label{Theorem: characterize Mobius} Let $\phi$ be a M\"obius transformation on $\overline{\Real^n}$. If $\phi(\infty) = \infty$, then $\phi$ is a Euclidean similarity. If $\phi(\infty) \ne \infty$, then there exists a unique reflection $\sigma$ in a sphere $S$ (necessarily centered at $\phi^{-1}(\infty)$) and a unique Euclidean isometry $\psi$ such that $\phi = \psi\sigma$.
\end{theorem}
In the statement of Theorem \ref{Theorem: characterize Mobius}, the sphere $S$ is called the \textit{isometric sphere} for the M\"obius transformation $\phi$. The main result of this paper is the following statement:

\begin{theorem}\label{Theorem: main result}
Let $\phi$ be a M\"obius transformation on $\overline{\Real^n}$ such that $b \coloneqq \phi^{-1}(\infty) \in \Real^n \backslash \{0\}$. Let $S^{n-1}(b, R)$ be its isometric sphere, $\sigma$ be the reflection in $S^{n-1}(b, R)$, and $a = \sigma(0)$. In addition, fix $\bar{r} \in (0, |b|)$ and let $f$ be the function on $\phi(B^n_{\bar{r}})$ defined by
$$f(x) = \frac{R^2|\phi^{-1}(x)|^2}{|b|^2 - |\phi^{-1}(x)|^2} \quad \text{for } x \in \phi(B^n_{\bar{r}}).$$
Suppose that $\Sigma$ is a $k$-dimensional minimal submanifold of $\phi(B^n_{\bar{r}})$ for some $0 < \bar{r} < |b|$ with $\partial\Sigma \subset \partial\phi(B^n_{\bar{r}})$. For $r \in (0, \bar{r})$, define
\begin{equation*}
J(r) = \left(\frac{|b|^2 - r^2}{r^2}\right)^\frac{k}{2} |\Sigma \cap \phi(B^n_r)| \;\; \text{and} \;\; I(r) = \left(\frac{|b|^2 - r^2}{r^2}\right)^\frac{k}{2} \int_{\Sigma \cap \phi(B^n_r)} \frac{|(x-\phi(0))^\top|^2}{|x-\phi(0)|^2}.
\end{equation*}
Then, for $0 < r < q < \bar{r}$, we have volume monotonicity
\begin{align}\label{eq: main result volume monotonicity}
J(q) - J(r) = R^k \int_{\Sigma \cap \phi(B^n_q) \backslash \phi(B^n_r)} f^{-\frac{k}{2}} \left( \frac{|b|^4|(x-\phi(0))^\perp|^2 + f^2|(\phi(\infty) - \phi(a))^\top|^2}{|b|^4|x-\phi(0)|^2}\right),
\end{align}
and weighted monotonicity
\begin{align}\label{eq: main result weighted monotonicity}
I(q) - I(r) &= R^k \int_{\Sigma \cap \phi(B^n_q) \backslash \phi(B^n_r)} f^{-\frac{k-4}{2}} \frac{|(\phi(\infty) - \phi(a))^\top|^2}{|b|^4|x-\phi(0)|^2} \nonumber \\
& \qquad \quad + k|b|^2 \int_{r}^{q} \left( \frac{(|b|^2 - \rho^2)^{\frac{k-2}{2}}}{\rho^{k+1}} \int_{\Sigma \cap \phi(B^n_{\rho})} \frac{|(x-\phi(0))^\perp|^2}{|x-\phi(0)|^2} \right) d\rho.
\end{align}
In particular, $I$ and $J$ are both monotone increasing on $(0, \bar{r})$. Moreover, if one of them is constant, then $\Sigma$ is a totally geodesic disk containing $\phi(0)$ and orthogonal to $\phi(\infty) - \phi(a)$.
\end{theorem}
Note that under the assumptions of Theorem \ref{Theorem: main result}, if $\Sigma_0$ is a $k$-dimensional flat plane containing $\phi(0)$ and orthogonal to $\phi(\infty) - \phi(a)$, then $\Sigma_0 \cap \phi(B^n_r)$ is a flat disk of radius $\frac{R^2r}{|b|\sqrt{|b|^2 - r^2}}$ for $0 < r < |b|$ (see Lemma \ref{Lemma: images of balls under reflections} and Theorem \ref{Theorem: characterize Mobius}).\\

Theorem \ref{Naff-Zhu Theorem} is a special case of Theorem \ref{Theorem: main result} when the M\"obius transformation is the reflection in a sphere orthogonal to the unit sphere. More specifically, for $a \in B^n_1 \backslash \{0\}$, we denote $a^* = \frac{a}{|a|^2}$ and consider the reflection $\sigma_a$ in the sphere $S^{n-1}(a^*, \sqrt{|a^*|^2 - 1})$. Then $\sigma_a(0) = a$ and the images of balls centered at the origin under this reflection satisfy (see Section \ref{Section: Preliminaries})
$$\sigma_a(B^n_r) = B^n\left(\frac{(1-r^2)a}{1-r^2|a|^2}\, ,\, \frac{(1-|a|^2)r}{1-r^2|a|^2}\right) \quad \text{for}\;\; 0 < r < \frac{1}{|a|}.$$
In other words, for $r \in (0, \frac{1}{|a|})$, we have
$$\sigma_a(B^n_r) = B^n((1-s)a,\, \sqrt{s(1-|a|^2) + s^2|a|^2}) \quad \text{for}\;\; s = \frac{(1-|a|^2)r^2}{1-r^2|a|^2}.$$
In addition, since $\sigma_a^{-1} = \sigma_a$, it is straightforward to check (or see the proof of Proposition \ref{Proposition: main result for reflections}) that if $\phi = \sigma_a$, then the functions $f$ in the statements of Theorem \ref{Theorem: main result} and Theorem \ref{Naff-Zhu Theorem} are indeed the same one.

Alternatively, one also recovers Theorem \ref{Naff-Zhu Theorem} from Theorem \ref{Theorem: main result} by considering the M\"obius transformation $\varphi_a$, $a \in B^n_1$, given by
$$\varphi_a(x) = (a + (1 - |a|^2)(a - x)^*)^* = \frac{|x-a|^2 a - (1-|a|^2)(x-a)}{1-2\langle a,x \rangle + |a|^2|x|^2}.$$
See, for example, \cite{Stoll_Harmonic and subharmonic function theory} for more details on $\varphi_a$.

\section{Preliminaries}\label{Section: Preliminaries}
In this section, we recall some useful properties of sphere reflections. For further details, we refer readers to \cite{Beardon_Geometry of Discrete}. Let $\overline{\Real^n} = \Real^n \cup \{\infty\}$ denote the one-point compactification of $\Real^n$. For $b \in \Real^n$ and $R > 0$, let $\sigma : \overline{\Real^n} \to \overline{\Real^n}$ be the reflection in the sphere $S^{n-1}(b,R)$, i.e., the map defined by
$$\sigma(b) = \infty, \quad \sigma(\infty) = b \quad \text{and} \quad \sigma(x) = b + \frac{R^2}{|x-b|^2} (x-b) \;\; \text{for}\;\; x \in \Real^n \backslash \{b\}.$$
It is a bijection on $\overline{\Real^n}$ which satisfies $\sigma(\sigma(x)) = x$. Moreover, by straightforward computation (or see \cite{Beardon_Geometry of Discrete}), for $x,y \in \Real^n \backslash \{b\}$, we have
\begin{equation}\label{eq:chord equation for reflections}
|\sigma(x) - b| = \frac{R^2}{|x-b|} \quad \text{and} \quad |\sigma(x) - \sigma(y)| = \frac{R^2|x-y|}{|x-b||y-b|}.
\end{equation}
In particular, assuming $b\ne 0$, by the first equation of \eqref{eq:chord equation for reflections}, we have $|\sigma(0)-b| = R^2/|b|$. Consequently, we choose $y = \sigma(0)$ in the second equation of \eqref{eq:chord equation for reflections} to obtain
\begin{equation}\label{eq: |sigma(x)|}
|\sigma(x)| = \frac{|b||x-\sigma(0)|}{|x-b|} \quad \text{for } x \in \Real^n \backslash \{b\} \text{ and assuming } b \ne 0.
\end{equation}
Explicit description of images of balls centered at the origin under sphere reflections is given in the following lemma.
\begin{lemma}\label{Lemma: images of balls under reflections}
Let $R > 0$, $b \in \Real^n \backslash \{0\}$ and $\sigma$ be the reflection in the sphere $S^{n-1}(b,R)$. Then
$$\sigma(S^{n-1}_r) =\left\{ \begin{array}{ll}
S^{n-1}\left( \dfrac{|b|^2 - R^2 -r^2}{|b|^2 - r^2}b, \; \dfrac{R^2r}{||b|^2 - r^2|}\right) & \quad \text{if }  r \ne |b| \smallskip \\
\{x \in \Real^n: 2|b|^2 - R^2 - 2\langle b,x \rangle = 0\} &\quad \text{if } r = |b|.
\end{array} \right.$$
\end{lemma}

\begin{proof}
We first denote $a = \sigma(0) = \frac{|b|^2 - R^2}{|b|^2}b$. Then, we replace $x$ by $\sigma(x)$ in \eqref{eq: |sigma(x)|} and use the fact that $\sigma(\sigma(x)) = x$ to get
$$\frac{|\sigma(x) - a|}{|\sigma(x) - b|} = \frac{|x|}{|b|} = \frac{r}{|b|} \quad \text{when}\;\, |x| = r.$$
The conclusion follows from a result in elementary geometry (known as Apollonius circle theorem).
\end{proof}

\begin{remark}\label{Remark: foliation of image of balls under reflection}
From Lemma \ref{Lemma: images of balls under reflections}, it is clear that if $\sigma$ is a reflection in $S^{n-1}(b,R)$ for $b \ne 0$, then the balls in the collection $\{\sigma(B^n_r), 0 < r < |b|\}$ foliate the half-space $\{x \in \Real^n: 2|b|^2 - R^2 - 2\langle b,x \rangle > 0\}$.
\end{remark}

\section{Proof of monotonicity formulas}

We divide the proof of Theorem \ref{Theorem: main result} into two steps. In the first step, we prove it under the assumption that the M\"obius transformation in the statement is a reflection. In the second step, we combine the result obtained in the first step and Theorem \ref{Theorem: characterize Mobius} to prove the main result in its full generality.
	
\subsection{The case of reflections}
In this subsection, we prove Theorem \ref{Theorem: main result} under the assumption that the M\"obius transformation in the statement is a reflection. More specifically, we prove the following statement:

\begin{proposition}\label{Proposition: main result for reflections}
Let $R > 0$, $b \in \Real^n \backslash \{0\}$ and $\sigma$ be the reflection in $S^{n-1}(b,R)$. Denote $a = \sigma(0)$. On the half-space $\{x \in \Real^n: 2|b|^2 - R^2 - 2\langle b,x \rangle > 0\}$, define the function
$$f(x) = \frac{R^2|\sigma(x)|^2}{|b|^2 - |\sigma(x)|^2} = \frac{|b|^2|x-a|^2}{2|b|^2 - R^2 - 2\langle b,x \rangle}.$$
Suppose that $\Sigma$ is a $k$-dimensional minimal submanifold of $\sigma(B^{n}_{\bar{r}})$ with $\partial \Sigma \subset \partial \sigma(B^{n}_{\bar{r}}) = \sigma(S^{n-1}_{\bar{r}})$ for some $0 < \bar{r} < |b|$. For $r \in (0, \bar{r})$, define
\begin{equation*}
J(r) = \left(\frac{|b|^2 - r^2}{r^2}\right)^\frac{k}{2} |\Sigma \cap \sigma(B^n_r)| \;\; \text{and} \;\; I(r) = \left(\frac{|b|^2 - r^2}{r^2}\right)^\frac{k}{2} \int_{\Sigma \cap \sigma(B^n_r)} \frac{|(x-a)^\top|^2}{|x-a|^2}.
\end{equation*} Then, for $0 < r < q < \bar{r}$, we have volume monotonicity
\begin{align}\label{eq: volume monotonicity for reflection}
J(q) - J(r) = R^k \int_{\Sigma \cap \sigma(B^n_q) \backslash \sigma(B^n_r)} f^{-\frac{k}{2}} \left( \frac{|b|^4|(x-a)^\perp|^2 + f^2|b^\top|^2}{|b|^4|x-a|^2}\right).
\end{align}
and weighted monotonicity
\begin{align}\label{eq: weighted monotonicity for reflection}
I(q) - I(r) &= R^k \int_{\Sigma \cap \sigma(B^n_q) \backslash \sigma(B^n_r)} f^{-\frac{k-4}{2}} \frac{|b^\top|^2}{|b|^4|x-a|^2} \nonumber \\
& \qquad \quad + k|b|^2 \int_{r}^{q} \left( \frac{(|b|^2 - \rho^2)^{\frac{k-2}{2}}}{\rho^{k+1}} \int_{\Sigma \cap \sigma(B^n_{\rho})} \frac{|(x-a)^\perp|^2}{|x-a|^2} \right) d\rho.
\end{align}
In particular, $I$ and $J$ are both monotone increasing on $(0, \bar{r})$. Moreover, if one of them is constant, then $\Sigma$ is a totally geodesic disk containing $a$ and orthogonal to $b$.
\end{proposition}
	
We give a proof of the above statement using coarea formula. This formula asserts (see e.g. \cite{Krantz-Parks_Geometric integration theory}) that if $f$ is a Lipschitz function on a manifold $\Sigma$, then for any locally integrable function $g$ on $\Sigma$ and for $s \leq t$, we have
$$\int_{\{s \leq f \leq t\}} g |\nabla^\Sigma f| = \int_{s}^{t} \left(\int_{\{f = \tau\}} g \right) d\tau.$$  

\begin{proof}[Proof of Proposition \ref{Proposition: main result for reflections}] Let $R > 0$, $b \in \Real^n \backslash \{0\}$ and $\sigma$ be the reflection in $S^{n-1}(b,R)$. Suppose that $\Sigma$ is a $k$-dimensional minimal submanifold of $\sigma(B^{n}_{\bar{r}})$ with $\partial \Sigma \subset \partial \sigma(B^{n}_{\bar{r}}) = \sigma(S^{n-1}_{\bar{r}})$ for some $0 < \bar{r} < |b|$. We denote
$$a = \sigma(0) = \frac{|b|^2 - R^2}{|b|^2}b.$$
Furthermore, on the half-space $P(b, R) \coloneqq \{x \in \Real^n: 2|b|^2 - R^2 - 2\langle b,x \rangle > 0\}$, define the function
\begin{equation}\label{eq: def of f in proof of main result}
f(x) = \frac{R^2|\sigma(x)|^2}{|b|^2 - |\sigma(x)|^2} \quad \text{for } x \in P(b,R).
\end{equation}
By Remark \ref{Remark: foliation of image of balls under reflection} and the fact that $\sigma^{-1} = \sigma$, we infer that $|\sigma(x)| < |b|$ for every $x \in P(b,R)$. Thus, $f$ is well-defined. Using \eqref{eq: |sigma(x)|}, we rewrite $f$ as
\begin{align}\label{eq: f rewritten}
f(x) &= \frac{R^2|x-a|^2}{|x-b|^2 - |x-a|^2} \nonumber \\
& = \frac{R^2|x-a|^2}{|b|^2 - |a|^2 - 2 \langle b-a, x \rangle} \nonumber \\
&= \frac{|b|^2|x-a|^2}{2|b|^2 - R^2 - 2\langle b,x \rangle} \quad \text{since } a = \frac{|b|^2 - R^2}{|b|^2}b .
\end{align}
On the other hand, by \eqref{eq: def of f in proof of main result} and using the fact $\sigma^{-1} = \sigma$ again, we have
$$f(x) = \frac{R^2r^2}{|b|^2 - r^2} \quad \text{if and only if } x \in \sigma({S^{n-1}_r}).$$
In other words, if we put $E_s = \{x: f(x) < s\}$, then
\begin{equation}\label{eq: E_s and sigma(Bnr)}
E_s = \sigma(B^n_r) \quad \text{whenever } s = \frac{R^2r^2}{|b|^2 - r^2}.
\end{equation}
To apply coarea formula, we first compute the gradient of $f$ on its level sets. We use \eqref{eq: f rewritten} to get
$$\log f = \log |b|^2 + \log |x-a|^2 - \log(2|b|^2 - R^2 - 2\langle b,x \rangle).$$
It follows that
\begin{equation*}
\frac{\nabla f}{f} = \frac{2(x-a)}{|x-a|^2} + \frac{2b}{2|b|^2 - R^2 -  2\langle b,x \rangle} = \frac{2}{|x-a|^2} \left( x - a + f(x) \frac{b}{|b|^2}\right).
\end{equation*}
So we have
\begin{equation*}
\nabla^\Sigma f = \frac{2s}{|x-a|^2} \left((x - a)^\top + s\frac{b^\top}{|b|^2}\right) \quad \text{on } \Sigma \cap \partial E_s.
\end{equation*}
The following argument is similar to one in \cite{Zhu_JFA2018}. For a fixed number $s > 0$, we consider the vector field
$$X_s = x - a - \frac{s}{|b|^2}b.$$
We have
\begin{equation}\label{eq: X_s with  Df on level sets}
\langle X_s, \nabla^\Sigma f \rangle = \frac{2s}{|x-a|^2} \left(|(x-a)^\top|^2 - \frac{s^2|b^\top|^2}{|b|^4}\right) \quad \text{on } \Sigma \cap \partial E_s.
\end{equation}
In addition, since $\div_\Sigma X_s = k$ and $\frac{\nabla^\Sigma f}{|\nabla^\Sigma f|}$ is the outward-pointing unit normal vector field to the region $\Sigma \cap E_s$ as a subset of $\Sigma$, we apply divergence theorem and \eqref{eq: X_s with  Df on level sets} to get
\begin{align}
|\Sigma \cap E_s| &= \frac{1}{k} \int_{\Sigma \cap E_s} \div_{\Sigma} X_s \nonumber \\
&= \frac{1}{k} \int_{\Sigma \cap \partial E_s}  \frac{\langle X_s, \nabla^\Sigma f\rangle}{|\nabla^\Sigma f|} \nonumber \\
&= \frac{2s}{k} \int_{\Sigma \cap \partial E_s} \frac{1}{|\nabla^\Sigma f| |x-a|^2} \left(|(x-a)^\top|^2 - \frac{s^2|b^\top|^2}{|b|^4}\right). \label{eq: |Sigma cap E_s| by coarea}
\end{align}
For $r \in (0, \bar{r})$, define
\begin{equation*}	J(r) = \left(\frac{|b|^2 - r^2}{r^2}\right)^\frac{k}{2} |\Sigma \cap \sigma(B^n_r)| \;\; \text{and} \;\; I(r) = \left(\frac{|b|^2 - r^2}{r^2}\right)^\frac{k}{2} \int_{\Sigma \cap \sigma(B^n_r)} \frac{|(x-a)^\top|^2}{|x-a|^2}.
\end{equation*}

We first prove monotonicity of $J$. To do this, we apply coarea formula to get
\begin{equation*}
|\Sigma \cap E_s| = \int_{0}^{s} \left( \int_{\Sigma \cap \partial E_\tau} \frac{1}{|\nabla^\Sigma f|} \right) d\tau.
\end{equation*}
Using this, we compute
\begin{align*}
\frac{d}{ds} \left( s^{-\frac{k}{2}}|\Sigma \cap E_s|\right) &= s^{-\frac{k}{2}} \frac{d}{ds}|\Sigma \cap E_s| - \frac{k}{2} s^{-\frac{k+2}{2}}|\Sigma \cap E_s|\\
&= s^{-\frac{k+2}{2}} \left(\int_{\Sigma \cap \partial E_s} \frac{s}{|\nabla^\Sigma f|} - \frac{k}{2} |\Sigma \cap E_s| \right).
\end{align*}
Consequently, we recall \eqref{eq: |Sigma cap E_s| by coarea} to obtain
\begin{equation*}
\frac{d}{ds} \left( s^{-\frac{k}{2}}|\Sigma \cap E_s|\right) = s^{-\frac{k}{2}} \int_{\Sigma \cap \partial E_s} \frac{1}{|\nabla^\Sigma f||x-a|^2}\left(|(x-a)^\perp|^2 + \frac{s^2|b^\top|^2}{|b|^4}\right).
\end{equation*}
Integrating the above equation and applying coarea formula again, we deduce
$$t^{-\frac{k}{2}}|\Sigma \cap E_t| - s^{-\frac{k}{2}}|\Sigma \cap E_s| = \int_{\Sigma \cap E_t \backslash E_s} \frac{f^{-\frac{k}{2}}}{|x-a|^2}\left(|(x-a)^\perp|^2 + \frac{f^2|b^\top|^2}{|b|^4}\right).$$
For $0 < r < q < \bar{r}$, we put $t = \frac{R^2q^2}{|b|^2 - q^2}$ and $s = \frac{R^2r^2}{|b|^2 - r^2}$. Then, we combine \eqref{eq: E_s and sigma(Bnr)} and the previous equation to conclude that
\begin{align}\label{eq: proof of monotonicity for J}
J(q) - J(r) &= R^k \left(t^{-\frac{k}{2}}|\Sigma \cap E_t| - s^{-\frac{k}{2}}|\Sigma \cap E_s| \right) \nonumber \\
&= R^k \int_{\Sigma \cap E_t \backslash E_s} \frac{f^{-\frac{k}{2}}}{|x-a|^2}\left(|(x-a)^\perp|^2 + \frac{f^2|b^\top|^2}{|b|^4}\right) \nonumber \\
&= R^k \int_{\Sigma \cap \sigma(B^n_q) \backslash \sigma(B^n_r)} \frac{f^{-\frac{k}{2}}}{|x-a|^2}\left(|(x-a)^\perp|^2 + \frac{f^2|b^\top|^2}{|b|^4}\right).
\end{align}

Now, we prove monotonicity of $I$. By coarea formula, we have
$$\int_{\Sigma \cap E_s} \frac{|(x-a)^\top|^2}{|x-a|^2} = \int_{0}^{s} \left( \int_{\Sigma \cap \partial E_\tau} \frac{|(x-a)^\top|^2}{|\nabla^\Sigma f||x-a|^2} \right) d\tau.$$
Using this, we compute
\begin{align*}
&\frac{d}{ds} \left(s^{-\frac{k}{2}} \int_{\Sigma \cap E_s} \frac{|(x-a)^\top|^2}{|x-a|^2} \right)\\
&= s^{-\frac{k}{2}} \frac{d}{ds}  \int_{\Sigma \cap E_s} \frac{|(x-a)^\top|^2}{|x-a|^2} - \frac{k}{2}s^{-\frac{k+2}{2}} \int_{\Sigma \cap E_s} \frac{|(x-a)^\top|^2}{|x-a|^2}\\
&= s^{-\frac{k+2}{2}} \left( \int_{\Sigma \cap \partial E_s} \frac{s|(x-a)^\top|^2}{|\nabla^\Sigma f||x-a|^2} - \frac{k}{2} \int_{\Sigma \cap E_s} \frac{|(x-a)^\top|^2}{|x-a|^2} \right)\\
&= s^{-\frac{k+2}{2}} \left( \int_{\Sigma \cap \partial E_s} \frac{s|(x-a)^\top|^2}{|\nabla^\Sigma f||x-a|^2} - \frac{k}{2}|\Sigma \cap E_s| + \frac{k}{2} \int_{\Sigma \cap E_s} \frac{|(x-a)^\perp|^2}{|x-a|^2} \right).
\end{align*}
Consequently, we recall \eqref{eq: |Sigma cap E_s| by coarea} to obtain
\begin{align*}
&\frac{d}{ds} \left(s^{-\frac{k}{2}} \int_{\Sigma \cap E_s} \frac{|(x-a)^\top|^2}{|x-a|^2} \right)\\
&\qquad = s^{-\frac{k-4}{2}} \int_{\Sigma \cap \partial E_s} \frac{|b^\top|^2}{|\nabla^\Sigma f||b|^4|x-a|^2} + \frac{k}{2}s^{-\frac{k+2}{2}} \int_{\Sigma \cap E_s} \frac{|(x-a)^\perp|^2}{|x-a|^2}.
\end{align*}
Integrating the above equation and applying coarea formula again, we find that
\begin{align*}
&t^{-\frac{k}{2}} \int_{\Sigma \cap E_t} \frac{|(x-a)^\top|^2}{|x-a|^2} - s^{-\frac{k}{2}} \int_{\Sigma \cap E_s} \frac{|(x-a)^\top|^2}{|x-a|^2}\\
&\qquad = \int_{\Sigma \cap E_t \backslash E_s} f^{-\frac{k-4}{2}} \frac{|b^\top|^2}{|b|^4|x-a|^2} + \frac{k}{2} \int_{s}^{t} \left( \tau^{-\frac{k+2}{2}} \int_{\Sigma \cap E_\tau} \frac{|(x-a)^\perp|^2}{|x-a|^2} \right) d\tau.
\end{align*}
From here, we apply change of variable $\tau = \frac{R^2\rho^2}{|b|^2 - \rho^2}$ to the last integral and use an argument similar to \eqref{eq: proof of monotonicity for J} to deduce \eqref{eq: weighted monotonicity for reflection}. 

Suppose either $I$ or $J$ is constant on $(0, \bar{r})$. Then we must have $(x-a)^\perp = 0$ and $b^\top = 0$ for every $x \in \Sigma$. This implies that $\Sigma$ is a totally geodesic disk containing $a$ and orthogonal to $b$.
\end{proof}

Similarly to \cite{Zhu_JFA2018}, we also have another proof of \eqref{eq: volume monotonicity for reflection} by applying divergence theorem to a special vector field. See also \cite{Brendle_GAFA2012,Freidin-McGrath_JFA2019,Freidin-McGrath_IMRN2020} where the method was applied to prove lower bound for volume of minimal submanifolds with free boundaries.

\begin{proof}[Alternative proof of \eqref{eq: volume monotonicity for reflection}] After starting the proof in the same way as in the first proof, we consider the vector field $W$ (similarly to \cite{Brendle-Hung_GAFA2017, Zhu_JFA2018}) on $\Sigma$ given by
$$ W(x) = \frac{1}{k} f(x)^{-\frac{k}{2}}(x-a) - F(f(x)) \frac{b}{|b|^2} \quad \text{for } x \in \Sigma,$$
where
$$F(t) =\left\{ \begin{array}{ll}
	\frac{1}{k-2} t^{-\frac{k-2}{2}} & \quad \text{if } k \geq 3, \smallskip \\
	-\frac{1}{2} \log t &\quad \text{if } k = 2.
	\end{array} \right.$$
Assuming $k \geq 3$, the explicit form of $W$ is
\begin{align*}
&W(x) = \frac{1}{k}  \left(\frac{2|b|^2 - R^2 - 2\langle b,x \rangle}{|b|^2|x-a|^2}\right)^{\frac{k}{2}}(x-a)\\
&\qquad \qquad \qquad \qquad \qquad - \frac{1}{k-2}\left(\frac{2|b|^2 - R^2 - 2\langle b,x \rangle}{|b|^2|x-a|^2}\right)^{\frac{k-2}{2}} \frac{b}{|b|^2}.
\end{align*}
We compute its divergence by
\begin{align*}
\div_\Sigma W(x) = &\left(\frac{2|b|^2 - R^2 - 2\langle b,x \rangle}{|b|^2|x-a|^2}\right)^{\frac{k}{2}} -  \left(\frac{2|b|^2 - R^2 - 2\langle b,x \rangle}{|b|^2|x-a|^2}\right)^{\frac{k-2}{2}} \frac{\langle (x-a)^\top, b^\top \rangle}{|b|^2|x-a|^2}\\
&-  \left(\frac{2|b|^2 - R^2 - 2\langle b,x \rangle}{|b|^2|x-a|^2}\right)^{\frac{k}{2}} \frac{|(x-a)^\top|^2}{|x-a|^2}\\
&+ \left(\frac{2|b|^2 - R^2 - 2\langle b,x \rangle}{|b|^2|x-a|^2}\right)^{\frac{k-4}{2}} \frac{|b^\top|^2}{|b|^4|x-a|^2}\\
&+  \left(\frac{2|b|^2 - R^2 - 2\langle b,x \rangle}{|b|^2|x-a|^2}\right)^{\frac{k-2}{2}} \frac{\langle (x-a)^\top, b^\top \rangle}{|b|^2|x-a|^2}.
\end{align*}
After cancelling repeated terms, we obtain the following simplified form
$$\div_\Sigma W(x) = f(x)^{-\frac{k}{2}} \frac{|(x-a)^\perp|^2}{|x-a|^2} + f(x)^{-\frac{k-4}{2}} \frac{|b^\top|^2}{|b|^4|x-a|^2}.$$
The previous computation also holds when $k = 2$. On the other hand, since $\partial E_s = \{f = s\}$, we apply divergence theorem together with the facts that $\div_\Sigma (x-a) = k$ and $\div_\Sigma b = 0$ to obtain
$$\int_{\Sigma \cap \partial E_s} \langle W, \nu \rangle = \int_{\Sigma \cap \partial E_s} \frac{1}{k} s^{-\frac{k}{2}} \langle x-a, \nu \rangle - \frac{F(s)}{|b|^2}\langle b, \nu \rangle = s^{-\frac{k}{2}}|\Sigma \cap E_s|,$$
where $\nu \coloneqq \frac{\nabla^\Sigma f}{|\nabla^\Sigma f|}$. Putting these together, we apply divergence theorem again to conclude that
\begin{align*}
t^{-\frac{k}{2}}|\Sigma \cap E_t| - s^{-\frac{k}{2}}|\Sigma \cap E_s| &= \int_{\Sigma \cap \partial E_t} \langle W, \nu \rangle - \int_{\Sigma \cap \partial E_s} \langle W, \nu \rangle\\
&= \int_{\Sigma \cap E_t \backslash E_s} \div_\Sigma W\\
&= \int_{\Sigma \cap E_t \backslash E_s}  f^{-\frac{k}{2}} \frac{|(x-a)^\perp|^2}{|x-a|^2} + f^{-\frac{k-4}{2}} \frac{|b^\top|^2}{|b|^4|x-a|^2}.
\end{align*}
We follow \eqref{eq: proof of monotonicity for J} to finish the proof.
\end{proof}

\subsection{General case}
In this subsection, we prove Theorem \ref{Theorem: main result}.
\begin{proof}[Proof of Theorem \ref{Theorem: main result}]
Let $\phi$ be a M\"obius transformation on $\overline{\Real^n}$ such that $b \coloneqq \phi^{-1}(\infty) \in \Real^n \backslash \{0\}$. According to Theorem \ref{Theorem: characterize Mobius}, we let $S^{n-1}(b, R)$ be the isometric sphere of $\phi$, $\sigma$ be the reflection in $S^{n-1}(b,R)$, and $\psi$ be the Euclidean isometry such that $\phi = \psi\sigma$. Denote $a = \sigma(0)$.

Suppose that $\Sigma$ is a $k$-dimensional minimal submanifold of $\phi(B^n_{\bar{r}})$ for some $0 < \bar{r} < |b|$ with $\partial\Sigma \subset \partial\phi(B^n_{\bar{r}})$. We first prove monotonicity of $J$, the function on $(0, \bar{r})$ defined by
$$J(r) = \left(\frac{|b|^2 - r^2}{r^2}\right)^\frac{k}{2} |\Sigma \cap \phi(B^n_r)|.$$
Since $\psi$ is an isometry on $\Real^n$, we infer that $\psi^{-1}(\Sigma)$ is a minimal submanifold of $\sigma(B^n_{\bar{r}})$ with $\partial\psi^{-1}(\Sigma) \subset \partial\sigma(B^n_{\bar{r}})$. Furthermore, we rewrite $J$ as
$$J(r) = \left(\frac{|b|^2 - r^2}{r^2}\right)^\frac{k}{2} | \psi^{-1}(\Sigma) \cap \sigma(B^n_r)|.$$
Applying Proposition \ref{Proposition: main result for reflections} to $\psi^{-1}(\Sigma)$, for $0 < r < q < \bar{r}$, we have
\begin{align}\label{eq: monotone of psi(Sigma)}
J(q) - J(r) = R^k \int_{\psi^{-1}(\Sigma) \cap \sigma(B^n_q) \backslash \sigma(B^n_r)} h^{-\frac{k}{2}} \left( \frac{|b|^4|(z-a)^\perp|^2 + h^2|b^\top|^2}{|b|^4|z-a|^2}\right),
\end{align}
where $h$ is the function defined by
$$h(z) \coloneqq \frac{R^2|\sigma(z)|^2}{|b|^2 - |\sigma(z)|^2} \quad \text{for } z \in \sigma(B^n_{\bar{r}}).$$
In addition, as a Euclidean isometry, $\psi = A + \psi(0)$ for some orthogonal matrix $A$. Thus, we apply change of variables $x = \psi(z)$ to \eqref{eq: monotone of psi(Sigma)} to deduce
\begin{align}\label{eq: proof of monotonicity for J general case}
J(q) - J(r) = R^k \int_{\Sigma \cap \phi(B^n_q) \backslash \phi(B^n_r)} f^{-\frac{k}{2}} \left( \frac{|b|^4|(x-\psi(a))^\perp|^2 + f^2|(\psi(b) - \psi(0))^\top|^2}{|b|^4|x-\psi(a)|^2}\right),
\end{align}
where $f$ is the function given by
$$f(x) \coloneqq h(\psi^{-1}(x)) =  \frac{R^2|\phi^{-1}(x)|^2}{|b|^2 - |\phi^{-1}(x)|^2} \quad \text{for } x \in \phi(B^n_{\bar{r}}).$$
Noting that $\psi(a) = \phi(0)$, $\psi(b) = \phi(\infty)$, and $\psi(0) = \phi(a)$, we derive (\ref{eq: main result volume monotonicity}) from (\ref{eq: proof of monotonicity for J general case}). The proof of (\ref{eq: main result weighted monotonicity}) is similar.
\end{proof}

\section*{Acknowledgements}	
The author was supported by National Key R\&D Program of China 2020YFA0712800. We would like to thank the anonymous referees for their helpful comments and suggestions.

\end{document}